\documentclass{amsart}
\usepackage {amsmath, amscd}
\usepackage {amssymb}
\usepackage {epsfig}
\usepackage {bm}
\usepackage {indentfirst} 
\usepackage{color}
\usepackage{mathpazo}

\usepackage{hyperref}

\numberwithin{equation}{section}

\def\<{\langle}
\def\>{\rangle}

\def\DD{{\mathcal D}}

\def\KK{{\mathcal K}}
\def\LL{{\mathcal L}}
\def\MM{{\mathcal M}}

\def\TT{{\mathcal T}}

\def\bbC{\mathbb{C}}

\def\bbD{\mathbb{D}}
\def\bbT{\mathbb{T}}

\def\1{\mathbf{1}}

\newcommand{\rank}{\mathop{\rm rank}}

\newcommand{\image}{\mathop{\rm Range}}

\newcommand{\MTT}{\MM\TT}

\newtheorem{lem}{Lemma}[section]
\newtheorem{thm}[lem]{Theorem}
\newtheorem{corollary}[lem]{Corollary}

\theoremstyle{definition}
\newtheorem{rmk}[lem]{Remark}

\newtheorem{example}[lem]{Example}

\title[Matrix valued truncated Toeplitz operators]{Matrix valued truncated Toeplitz operators: basic properties}

\author{Rewayat Khan}
\address{Abdus Salam School of Mathematical Sciences, GC University, Lahore, Pakistan}
\email{rewayat.khan@gmail.com}

\author{Dan Timotin}
\address{Institute of Mathematics Simion Stoilow of the Romanian Academy, Calea Grivitei 21, Bucharest, Romania}
\email{dan.timotin@imar.ro}


\begin{document}
	\maketitle

\begin{abstract}
	 Matrix valued truncated Toeplitz operators act on vector-valued model spaces. They represent a generalization of block Toeplitz matrices. A characterization of these operators analogue to the scalar case is obtained, as well as the determination of the symbols that produce the zero operator.
\end{abstract}

\section{Introduction}

The Toeplitz operators are compressions of multiplication operators on the space $L^2(\bbT)$  to the Hardy space $H^2$. With respect to the standard exponential basis, their matrices are constant along diagonals; if we truncate such a matrix considering only its upper left finite corner, we obtain classical Toeplitz matrices.

A great deal of attention in the last decade has been attracted by certain   generalizations of Toeplitz matrices, namely compressions of multiplication operators to subspaces of the Hardy space which are invariant under the backward shift. These  ``model spaces'' are of the form $H^2\ominus uH^2$ with $u$ an inner function, and the compressions are called truncated Toeplitz operators. They have been formally introduced in~\cite{Sa}; see~\cite{GR} for a more recent survey. Although classical Toeplitz matrices have often been a starting point for investigating truncated Toeplitz operators, the latter have a much richer and more interesting theory.

In the theory of contractions on a Hilbert space, these model spaces are the scalar case of a more general construction, which provides functional models for arbitrary completely nonunitary contractions. In particular, 
it makes sense to consider matrix-valued innner functions $\Theta$ and the associated model space $K_\Theta=H^2(E)\ominus \Theta H^2(E)$, with $E$ a finite dimensional Hilbert space.

We  develop below the basics of the corresponding matrix-valued truncated Toeplitz operators, which are compressions to $K_\Theta$ of multiplications with matrix-valued functions on $H^2(E)$. 
From an alternate point of view, these operators are generalizations of finite block Toeplitz matrices. 
With respect to the exposition in~\cite{Sa}, one sees that different new questions have to be addressed, mostly related to the noncommutativity of  matrices. 

The structure of the paper is the following.
 After a section of general preliminaries about spaces of vector and matrix valued functions, we give a primer of the properties of the vector-valued  model space. Matrix-valued truncated Toeplitz operators are introduced in Section~\ref{se:MTTO}, where we discuss the main specific difficulties that appear. The next two sections contain the main results of the paper: two intrinsic characterizations of these operators and the identification of symbols that correspond to the null operator. In the final section we determine a class of finite rank matrix-valued truncated Toeplitz operators.

\section{Preliminaries}


Let $\mathbb{C}$ denote the complex plane, $\mathbb{D}={\{z\in\mathbb{C}: |z|<1}\}$ the unit disc, $\mathbb{T}={\{z\in\mathbb{C}:|z|=1}\}$ the unit circle.
Throughout the paper $E$ will denote a fixed Hilbert space,  of finite dimension $d$, and $\LL(E)$ the algebra of bounded linear operators on $E$, which  may be identified with $d\times d$ matrices. Part of the development below may be carried through for $ E $ an infinite dimensional Hilbert space; we will however restrict ourselves to  $ \dim E<\infty $, avoiding certain delicate problems of convergence.

The space $L^{2}(E)$  is defined,  as usual, by
\[
L^{2}(E)={\Big\{F:\mathbb{T}\to E: F(\zeta)=\sum\limits_{-\infty}^{\infty}a_{n}e^{int}:  a_{n}\in E,\quad   \sum\limits_{-\infty}^{\infty}\|a_{n}\|^{2}<\infty }\Big\},
\]
endowed with the inner product
\begin{equation*}
\langle F,G\rangle_{L^{2}(E)}=\frac{1}{2\pi}\int\limits_{0}^{2\pi}\langle F(e^{it}),G(e^{it})\rangle_{E}\,dt.
\end{equation*}
If $\dim E=1$ (i.e $E=\mathbb{C}$) then $L^{2}(E)$ consists of scalar-valued functions and is denoted by $L^{2}$.

The space
\[
L^{\infty}(\mathcal{L}(E))=\{F:\mathbb{T}\to \mathcal{L}(E): \emph{F \emph{is measurable and bounded}}\}
\]
acts on $L^2(E)$ by means of multiplication: to $ \Phi\in L^{\infty}(\mathcal{L}(E)) $ we associate the operator $ M_\Phi $ defined by $ M_\Phi(F)=\Phi F $. 

By viewing $\LL(E)$ as a Hilbert space (endowed with the Hilbert--Schmidt norm), one can also consider the space $L^{2}(\mathcal{L}(E))$, which may be identified with
matrices with all the entries in $L^{2}(\mathbb{T})$. In particular, $ L^{\infty}(\mathcal{L}(E))\subset L^{2}(\mathcal{L}(E)) $.  Alternately, we may view $L^2(\LL(E))$ also as a space of square summable Fourier series with coefficients in $L^2(E)$.

The Hardy space  $H^{2}(E)$ is the subspace of $L^2(E)$ formed by the functions with vanishing negative Fourier coefficients; it can be identified with a space of $E$-valued  functions analytic in $\bbD$, from which  the boundary values can be recovered almost everywhere through radial limits. One can also view  $H^{2}(E)$ as the direct sum of $d$ standard $H^{2}$ spaces.
We have an
orthogonal decomposition
\[
L^{2}(\mathcal{L}(E))=[zH^{2}(\mathcal{L}(E))]^{*}\oplus H^{2}(\mathcal{L}(E)).
\]

Let $S$ denote the forward shift operator $(Sf)(z)=zf(z)$ on $H^{2}(E)$; it is the restriction of $ M_z  $ to $H^{2}(E)$. Its adjoint (the backward shift) is the operator \[(S^{*}f)(z)=\frac{f(z)-f(0)}{z}.
\]
One sees easily that $I-SS^*$ is precisely the orthogonal projection onto the space of constant functions.


An \emph{inner function} will be an element $\Theta\in H^2(\LL(E))$ whose boundary values are almost everywhere unitary operators in $\LL(E)$. The following lemma is a consequence of more general results about factorization of analytic operator valued functions (see~\cite{NF}).

\begin{lem}\label{le:inner functions}
	If $\Theta$ is an inner function, $\Theta_1$ is a bounded analytic function,  and $\Theta\Theta_1$ is constant, then $\Theta$ is also constant.
\end{lem}

We will also use the following simple result.

\begin{lem}\label{le:from functions to matrices - Theta}
	Suppose $\Phi\in H^2(\LL(E))$, and $\Phi(z)x\in \Theta H^2(E)$ for any $x\in E$. Then there exists $\Phi_1\in H^2(\LL(E))$ such that $\Phi=\Theta\Phi_1$.
\end{lem}

\begin{proof}
	Let $\{e_1, \dots, e_d\}$ be an orthonormal basis in $E$. Then we may take as $\Phi_1$ the matrix having as columns $\Phi(z)e_1, \dots, \Phi(z)e_d$.
\end{proof}

The  \emph{model space} associated to $\Theta$, denoted by $\mathcal{K}_{\Theta}$,  is defined by $\mathcal{K}_{\Theta}=H^{2}(E)\ominus \Theta H^{2}(E)$; the orthogonal projection onto $\KK_\Theta$ will be denoted by $P_\Theta$. The properties of the model space are familiar to many analysts in the scalar case.  On the other hand, the vector valued version is less widely known (despite playing an important role in the Sz.-Nagy--Foias theory of contractions~\cite{NF}); the next section will be a primer of its main properties.

From the point of view of the theory of contractions, the spaces $\KK_\Theta$ represent models for contractions $T$ with $\rank(I-T^*T)= \rank(I-TT^*)=d$ and $T^n\to 0$ strongly. We will not pursue this point of view, which is extensively developed in~\cite{NF}; rather, we will discuss the model space as an intrinsic functional object.

\section{Properties of the model space}\label{se:model space}

For the development in this section we refer to~\cite{Fu}, where  the context is that of model spaces for completely nonunitary contractions.
We will always suppose that the inner function $\Theta$ is \emph{pure}, which means that $\|\Theta(0)\|<1$.

%

The
 model space $\mathcal{K}_{\Theta}$  is  a vector valued reproducing kernel Hilbert space; its reproducing  kernel function, which takes values in $\LL(E)$, is
 \[
 k_{\lambda}^{\Theta}(z)=
 \frac{1}{1-\overline{\lambda}z}(I-\Theta(z)\Theta(\lambda)^{*}).
 \]
 This means that for any $x\in E$ we have $k_{\lambda}^{\Theta}x\in \mathcal{K}_{\Theta}$, and, if $F\in \KK_\Theta$, then
 \[
 \langle F, k_{\lambda}^{\Theta}x\rangle_{\KK_\Theta}=\< F(\lambda),x \>_E.
 \]

%
 If  $\Theta(z)\in H^2(\LL(E))$ we define the new function $\widetilde{\Theta}(z)=\Theta(\overline{z})^{*}$. Then $\widetilde{\Theta}(z)$ is inner if and only if $\Theta(z)$ is inner. In this case the operator $\tau:L^{2}(E)\to L^{2}(E)$ defined by
 \begin{equation}\label{eq:tau}
 (\tau f)(e^{it})=e^{-it}\Theta(e^{-it})^{*}f(e^{-it}),
 \end{equation}
  is unitary and $\tau(\mathcal{K}_{\Theta})=\mathcal{K}_{\widetilde{\Theta}}$; thus $\tau P_{\Theta} =P_{\mathcal{K}_{\widetilde{\Theta}}}\tau$.
 The adjoint of $\tau$ is given by
 \[
 (\tau^{*}f)(e^{it})=e^{-it}\Theta(e^{it}) f(e^{-it}).
 \]

We have already met a class of elements in $\KK_\Theta$, namely the functions $k_\lambda^\Theta x$ for $x\in E$. Another related family is obtained by transporting through $\tau$ the reproducing kernels in $ \mathcal{K}_{\widetilde{\Theta}} $. So we define, for  $x\in E$,  $\widetilde{k_{\lambda}^{\Theta}}x=\tau^* (k_\lambda^{\widetilde{\Theta}}x)$; computations give
\[
\widetilde{k_{\lambda}^{\Theta}}
x=
\frac{1}{z-\lambda}\big(\Theta(z)-\Theta(\lambda)\big) x.
\]

The \emph{model operator} $S_{\Theta}\in\LL(\mathcal{K}_{\Theta})$ is defined  by the formula
\[
(S_{\Theta}f)(z)=P_{\Theta} (zf), \quad f\in \mathcal{K}_{\Theta}.
\]
The adjoint of $S_{\Theta}$  is given by
\[
(S_{\Theta}^{*}f)(z)=\frac{f(z)-f(0)}{z};
\]
it is the restriction of the left shift in $H^{2}(E)$ to the $S^{*}$-invariant subspace $\mathcal{K}_{\Theta}$. We will also use the formula
\begin{equation}\label{eq:comm S_Theta and tau}
\tau S_{\Theta}=S_{\widetilde{\Theta}}^{*}\tau.
\end{equation}

 The action of $S_\Theta$ is more precisely described if we introduce the following subspaces of $\mathcal{K}_{\Theta}$ (the defect spaces of $ S_\Theta $ in the terminology of~\cite{NF}):
\begin{equation}\label{eq:definition of defect spaces}
\begin{split}
\mathcal{D}&=\{(I-\Theta(z)\Theta(0)^{*})x:x\in E\}\\
\widetilde{\mathcal{D}}&=\Big\{\frac{1}{z}(\Theta(z)-\Theta(0))x:x\in E \Big\}.
\end{split}
\end{equation}
Then the following relations hold:
\begin{equation}\label{eq:action of S_Theta}
\begin{split}
(S_{\Theta}f)(z)&=\begin{cases}
z f(z)&\text{ for }f\perp \widetilde{\mathcal{D}},\\
-(I-\Theta(z)\Theta(0)^{*})\Theta(0)x &\text{ for }f=\frac{1}{z}(\Theta(z)-\Theta(0))x\in\widetilde{\DD};
\end{cases}\\
(S^{*}_{\Theta}f)(z)&=
\begin{cases}
\frac{f(z)}{z}&\text{ for }f\perp {\mathcal{D}},\\
-\frac{1}{z}(\Theta(z)-\Theta(0))\Theta(0)^{*}x. &\text{ for }f=(I-\Theta(z)\Theta(0)^{*})x\in {\mathcal{D}},
\end{cases}
\end{split}
\end{equation}
\begin{equation}\label{eq:inclusions S(D), etc}
S_{\Theta}(\widetilde{\mathcal{D}})\subset\mathcal{D},
\quad
S_{\Theta}(\widetilde{\mathcal{D}}^\perp)\subset\mathcal{D}^\perp,
\quad
S_{\Theta}^{*}(\mathcal{D})\subset \widetilde{\mathcal{D}}, \quad S_{\Theta}^{*}(\mathcal{D}^\perp)\subset \widetilde{\mathcal{D}}^\perp,
\end{equation}
\begin{equation}\label{eq:image of defects}
\image (I-S_{\Theta}S_{\Theta}^{*})=\DD, \quad \image (I-S_{\Theta}^*S_{\Theta})=\widetilde{\DD},
\end{equation}
and
\begin{equation}\label{eq:I-S_Theta S_Theta*}
(I-S_{\Theta}S_{\Theta}^{*})f=(I-\Theta(z)\Theta(0)^{*})f(0).
\end{equation} 
 
From~\eqref{eq:image of defects} it follows that there are operators $J, \tilde{J}\in\LL(\KK_\Theta)$, such that 
\begin{align}
P_\DD&= (I-S_\Theta S_\Theta^*)J=J^*  (I-S_\Theta S_\Theta^*), \label{eq:J}  \\
P_{\widetilde{\DD}}&=  (I-S_\Theta^* S_\Theta)\tilde J=\tilde J^*  (I-S_\Theta^* S_\Theta).\label{eq:tilde J}
\end{align}

 Since $(I_{H^{2}(E)}-SS^{*})f(z)=f(0)$, $(I_{\mathcal{K}_{\Theta}}-S_{\Theta}S_{\Theta}^{*})f(z)=(I-\Theta(z)\Theta(0)^{*})f(0)$, and $(I-\Theta(z)\Theta(0)^{*})^{-1}\in H^\infty(\LL(E))$, we may define the operator
 $\Omega:\DD\to E$  by
 \begin{equation}\label{eq:omega}
 \Omega (I_{\mathcal{K}_{\Theta}}-S_{\Theta}S_{\Theta}^{*})f=(I_{H^{2}(E)}-SS^{*})f.
 \end{equation}

The next result is the analog of~\cite[Lemma 2.2]{Sa}.

\begin{lem}\label{le:formulas for S_Theta on k lambda}
	For $\lambda\in \mathbb{D}$ and $x,y\in E$ constant functions in $H^{2}(E)$ we have
	\begin{equation*}
	\begin{split}
S_{\Theta}(k_{\lambda}^{\Theta}x)&=\frac{1}{\overline{\lambda}}k_{\lambda}^{\Theta}x-\frac{1}{\overline{\lambda}}k_{0}^{\Theta}x \text{\quad for } \lambda\not=0,\\
S_{\Theta}(\widetilde{k_{\lambda}^{\Theta}}y) &=\lambda\widetilde{k_{\lambda}^{\Theta}}y-k_{0}^{\Theta}\Theta(\lambda)y.
	\end{split}
	\end{equation*}
\end{lem}

\begin{proof}
	
	Since $\Theta(z)\Theta(\lambda)^{*}\frac{1}{1-\overline{\lambda}z}x\in \Theta H^{2}(E)$, we have
	\begin{align*}
	S_{\Theta}k_{\lambda}^{\Theta}x
	&=P_{\Theta} (zk_{\lambda}^{\Theta}x)
	=P_{\Theta} \bigg(\frac{z}{1-\overline{\lambda}z}(I-\Theta(z)\Theta(\lambda)^{*})x\bigg) \\
	&=P_{\Theta} \big(\frac{z}{1-\overline{\lambda}z}x\big)-P_{\Theta} \big(\Theta(z)\Theta(\lambda)^{*}\frac{1}{1-\overline{\lambda}z}\big)x
	=P_{\Theta} (\frac{z}{1-\overline{\lambda}z}x)\\
	&=P_{\Theta} \bigg(\frac{1}{\overline{\lambda}}\big(\frac{1}{1-\overline{\lambda}z}-1\big)x\bigg)
	=\frac{1}{\overline{\lambda}}k_{\lambda}^{\Theta}x-\frac{1}{\overline{\lambda}}k_{0}^{\Theta}x.
	\end{align*}

	For the second equality, we use  $P_{\Theta} (\Theta(z)y)=0$ to obtain
	\begin{align*}
	S_{\Theta}(\widetilde{k_{\lambda}^{\Theta}}y)
	&=P_{\Theta} (z\widetilde{k_{\lambda}^{\Theta}})y
	=P_{\Theta} \bigg(\frac{z}{z-\lambda}\big(\Theta(z)-\Theta(\lambda)\big) y\bigg)\\
	&=P_{\Theta} \big(\Theta(z)-\Theta(\lambda)\big)y +P_{\Theta} \bigg(\frac{\lambda}{z-\lambda}\big(\Theta(z)-\Theta(\lambda)\big)y\bigg)\\
	&=-P_{\Theta} (\Theta(\lambda)y)+\lambda P_{\Theta} \widetilde{k_{\lambda}^{\Theta}}y=\lambda\widetilde{k_{\lambda}^{\Theta}}y-k_{0}^{\Theta}\Theta(\lambda)y. \qedhere
	\end{align*}
\end{proof}

\section {Matrix valued truncated Toeplitz operators}\label{se:MTTO}
Suppose $\Theta$ is a fixed pure inner function.
Since the space  $\mathcal{K}_{\Theta}$ is spanned by the functions $k^\Theta_\lambda x$, $\lambda\in\bbD$, $x\in E$, which are bounded, it follows that the subspace $\mathcal{K}_{\Theta}^{\infty}=\mathcal{K}_{\Theta}\cap H^{\infty}(E)$ of all bounded functions in $\mathcal{K}_{\Theta}$ is dense in $\mathcal{K}_{\Theta}$.

 Suppose now that $\Phi\in L^{2}(\mathcal{L}(E))$. Consider the linear map $f\mapsto P_{\Theta}(\Phi f)$, defined on $\mathcal{K}_{\Theta}^{\infty}$. In case it is bounded, it uniquely determines an operator in $\LL(\KK_\Theta)$, denoted by $A_\Phi^\Theta$, and called a  \emph{matrix-valued  truncated Toeplitz operator} (MTTO).
The function $\Phi$ is then called a \emph{symbol} of the operator. We will usually drop the superscript $\Theta$, as we consider a fixed inner function. We denote  by $\MTT(\KK_\Theta)$ the space of all MTTOs on the model space $\KK_\Theta$.

In the particular case $\Theta(z)=z^n I_E$, the MTTOs obtained are actually familiar objects, namely block Toeplitz matrices of dimension $n$, in which the entries are matrices of dimension $d$. They have have been extensively studied in linear algebra and related areas (see, for instance,~\cite{BS}).

It is immediate that 
\begin{equation}\label{eq:adjoint}
A_\Phi^*=A_{\Phi^*};
\end{equation}
so $\MTT(\KK_\Theta)$ is a selfadjoint linear  space.

The operator $S_\Theta$ is a simple example of a MTTO; it is obtained by taking $\Phi(z)=z I_E$. This example is rather special because the symbol is scalar-valued.

Obviously
MTTOs may be viewed as  matrix valued analogues of the scalar truncated Toeplitz operators introduced by Sarason in~\cite{Sa}. However, that theory cannot be extended smoothly, with analogous proofs; there are several difficulties that one encounters from the very beginning and that  we will point out next.

First, although the space $\Theta H^{2}(E)\subset H^2(E)$ is invariant with respect to $S=M_z$, it is not  invariant for $M_\Phi$ for a general analytic $\Phi$, and consequently $\KK_\Theta$ is not invariant with respect to $M_\Phi^*$; that is, we do not have the relation $A_\Phi^*=M_\Phi^*|\KK_\Theta$. This remark is the main source of difficulties in extending the theory from the scalar to the matrix valued case, so it is useful to illustrate it by a simple example.

\begin{example}\label{ex:1}
	Take $d=2$, and
	\[
	\Theta(z)=
	\begin{pmatrix}
	z&0\\0&z^2
	\end{pmatrix}, \quad
	\Phi(z)=
	\begin{pmatrix}
	0&0\\1&0
	\end{pmatrix}.
	\]
Then 
\[
\Theta H^2(E)=\left\{\begin{pmatrix}
	z f(z)\\ z^2 g(z)
	\end{pmatrix} : f,g\in H^2  \right\}.
\]
So
\[
\begin{pmatrix}
z\\ 0
\end{pmatrix}=\Theta(z)\begin{pmatrix}
1\\ 0
\end{pmatrix}\in \Theta H^2(E)\qquad\text{but\quad} 
\begin{pmatrix}
0\\ z
\end{pmatrix}=\Phi(z)\begin{pmatrix}
z\\ 0
\end{pmatrix}\not\in \Theta H^2(E).
\]	
		\end{example}

This difficulty does not appear in an important particular case. 
It is easy to prove that if $\Phi\in H^2(\LL(E))$ and there exists  $\Phi_{1}\in H^{2}(\mathcal{L}(E))$ such that 
\begin{equation}\label{eq:theta phi=phi' theta}
\Phi\Theta=\Theta\Phi_{1},
\end{equation}
  then $\Theta H^{2}(E)$ is invariant with respect to $M_\Phi$. (In particular, this happens when $\Phi$ commutes with $\Theta$.) Then $M_\Phi^*\KK_\Theta\subset \KK_\Theta$ and therefore $A_\Phi^*=M_\Phi^*|\KK_{\tilde \Theta}\Theta$. It follows that $A_\Phi^*S_\Theta^*=S_\Theta^*A_\Phi^*$, and therefore $A_\Phi S_\Theta=S_\Theta A_\Phi$.

According to the lifting commutant theorem of Sz-Nagy and Foias (see~\cite[Chapter VI]{NF}), the converse is also valid; namely, if $A\in\LL(E)$ and $AS_\Theta= S_\Theta A$, then there exists $\Phi$, even in $H^\infty(\LL(E))$, such that~\eqref{eq:theta phi=phi' theta} is valid and $A=A_\Phi$. 
A similar result holds by passing to the adjoint and using~\eqref{eq:adjoint}.  The next theorem yields then the first large class of MTTOs.

\begin{thm}\label{th.commutant and *commutant}
	The linear space $\{S_\Theta\}'+\{S_\Theta^*\}'$ is contained in $ \MTT(\KK_\Theta) $.
\end{thm}

Even in the scalar case, the inclusion is in general strict. A simple argument is the fact that, as noted above, the operators in $\{S_\Theta\}'+\{S_\Theta^*\}'$ always have bounded symbols, which is known not to be in general the case (see~\cite{BBK, BCFMT}).

Another difference from the scalar valued case stems from the nonexistence of a canonical conjugation. 
Remember that a \emph{conjugation} on a Hilbert space is a conjugate-linear, isometric and involutive map. If $C$ is a conjugation, then a bounded linear operator $T$  is called \emph{$C$-symmetric} if $T=CT^{*}C$ (see~\cite{GP}).

In the scalar case there exists a canonical conjugation with respect to which all truncated Toeplitz operators are symmetric (see~\cite{Sa}). This is no longer true in our case; actually, it follows from results in~\cite{CFT} that the model operator $S_\Theta$ is complex symmetric if and only if there exists a conjugation $\Gamma$  on $E$, with the property that for all $z\in\bbD$ the matrix $\Theta(z)$ is $\Gamma$-symmetric (or, equivalently, $ \Theta(e^{it})$ is $\Gamma$-symmetric a.e on $\bbT$). In that case  $C_\Gamma$  on $L^{2}(E)$ defined by  $C_\Gamma f=\Theta \bar{z}\Gamma f$ is a conjugation on $L^2(E)$, that leaves  $\KK_\Theta$ invariant, and $S_\Theta$ is $C_\Gamma$-symmetric. 
However, even if such a conjugation $\Gamma$ exists, not all MTTOs are $C_\Gamma$-symmetric. For instance, one can check in Example~\ref{ex:1} that $\Gamma$ defined on $\bbC^2$ by $\Gamma(a_1, a_2)=(\bar a_1, \bar a_2)$ is a conjugation such that $\Theta(z)$ is $\Gamma$-symmetric for all $z\in\bbD$, but 
$A_\Phi$ is not symmetric with respect to the corresponding $C_\Gamma$. The most we can obtain is the following statement, whose proof is straightforward.

\begin{thm} Let $\Gamma$ be is a conjugation on $E$, and $C_\Gamma f=\Theta \bar{z}\Gamma f$. Suppose that a.e on $\mathbb{T}$ $\Theta(e^{it})$ and $\Phi(e^{it})$ are $\Gamma$-symmetric  a.e on $\mathbb{T}$, and $\Phi(e^{it}) \Theta(e^{it})= \Theta(e^{it})\Phi(e^{it})$. Then $A_\Phi$ is $C_\Gamma$-symmetric.
\end{thm}


%
%
%

\section {Characterization of Matrix Valued Truncated Toeplitz Operators}

We  obtain characterizations of MTTOs similar to those obtained in the scalar case by Sarason in~\cite[Theorem 4.1 and 8.1]{Sa}.
We start by recalling that for $\dim E=1$ we have $S_\Theta A_\Phi=A_\Phi S_\Theta$, and therefore 
\[
A_\Phi-S_\Theta A_\Phi S_\Theta^*=
A_\Phi- A_\Phi S_\Theta S_\Theta^*=
A_\Phi(I-  S_\Theta S_\Theta^*) . 
\]
This useful formula is not true in the vector valued context; however, the next lemma provides a useful replacement.

\begin{lem}\label{le:semi-commutator for analytic}
If $\Phi\in H^2(\LL(E))$ then
\begin{equation}\label{eq:semi-commutator for analytic}
A_\Phi-S_\Theta A_\Phi S_\Theta^*= P_\Theta M_\Phi(I-SS^*) | \KK_\Theta.
\end{equation}
\end{lem}

\begin{proof}
	Since $M_z \Theta H^2(E)\subset \Theta H^2(E)$, we have $P_\Theta M_zP_\Theta=P_\Theta M_z$. Therefore
	\[
	\begin{split}
	A_\Phi-S_\Theta A_\Phi S_\Theta^*
	&=P_\Theta M_\Phi P_\Theta - P_\Theta M_z P_\Theta M_\Phi	M_z^*P_\Theta\\
&	=P_\Theta(M_\Phi-M_zM_\Phi M_z^*)P_\Theta
=P_\Theta(M_\Phi-M_\Phi M_z M_z^*)P_\Theta\\
&=P_\Theta M_\Phi(I-SS^*)P_\Theta. \qedhere
	\end{split}
	\]
\end{proof}

\begin{thm}\label{th:rank two characterization of MTTO}
	The bounded operator $A$ on $\mathcal{K}_{\Theta}$ belongs to $\mathcal{T}(\mathcal{K}_{\Theta})$ if and only if
	\begin{equation}\label{eq:quasicommutation}
	A-S_{\Theta}AS_{\Theta}^{*}=B(I-S_\Theta S_\Theta^*)+(I-S_\Theta S_\Theta^*) B^{\prime *},
	\end{equation}
	for some operators $B, B^{\prime}$ from $\DD$ to $\mathcal{K}_{\Theta}$.
\end{thm}

\begin{proof}
	Suppose that A is a bounded operator on $\mathcal{K}_{\Theta}$ that satisfies~\eqref{eq:quasicommutation}. We have then for any positive integer
	  $n$
	  \[
	  S_{\Theta}^{n}AS^{*n}_{\Theta}-S_{\Theta}^{n+1}AS^{*n+1}_{\Theta}=S_{\Theta}^{n}B(I-S_\Theta S_\Theta^*) S_{\Theta}^{*n}+S_{\Theta}^{n}(I-S_\Theta S_\Theta^*) B^{\prime*}S_{\Theta}^{*n}
	  \]
	  and, adding for $n=0,1,2,...,N$,
	\[
	A=\sum\limits_{n=0}^{N}[S_{\Theta}^{n}B(I-S_\Theta S_\Theta^*) S_{\Theta}^{*n}+S_{\Theta}^{n}(I-S_\Theta S_\Theta^*) B^{\prime*}S_{\Theta}^{*n}]+S_{\Theta}^{N+1}AS^{*N+1}_{\Theta}.
	\]  
	Take $f,g\in \mathcal{K}_{\Theta}^\infty$; then
	\[
	\begin{split}
	\langle Af,g\rangle&=\sum\limits_{n=0}^{N}[\langle S_{\Theta}^{n}B(I-S_\Theta S_\Theta^*) S_{\Theta}^{*n}f,g\rangle+\langle S_{\Theta}^{n}(I-S_\Theta S_\Theta^*) B^{\prime*}S_{\Theta}^{*n}f,g\rangle]\\
	&\qquad+\langle AS^{*N+1}_{\Theta}f,S_{\Theta}^{*N+1}g\rangle.
	\end{split}
	\]
	Since $S_{\Theta}^{*N}\to 0$ strongly as $N\to \infty$ we obtain
	\begin{equation*}
	\langle Af,g\rangle=\sum\limits_{n=0}^{\infty}[\langle S_{\Theta}^{n}B(I-S_\Theta S_\Theta^*) S_{\Theta}^{*n}f,g\rangle+\langle S_{\Theta}^{n}(I-S_\Theta S_\Theta^*) B^{\prime*}S_{\Theta}^{*n}f,g\rangle].
	\end{equation*}
or
	\begin{equation}\label{eq:<Af,g>}
\langle Af,g\rangle=\sum\limits_{n=0}^{\infty}[\langle S_{\Theta}^{n}B(I-S_\Theta S_\Theta^*) S_{\Theta}^{*n}f,g\rangle+\langle S_{\Theta}^{*n}f,B^{\prime}(I-S_\Theta S_\Theta^*) S_{\Theta}^{*n}g\rangle].
\end{equation}

Suppose now $f=\sum_{k=0}^\infty a_n z^n$, with $a_n\in E$. Then $S_{\Theta}^{*n}f=S^{*n}f=\sum\limits_{k=0}^{\infty}a_{n+k}z^{k}$ and, according to~\eqref{eq:I-S_Theta S_Theta*}, 
$
(I-S_\Theta S_\Theta^*) S_{\Theta}^{*n} f= (I-\Theta(z)\Theta(0)^{*})a_{n}$.
Similarly, if $g=\sum_{k=0}^\infty b_n z^n$, with $b_n\in E$, then 
$(I-S_\Theta S_\Theta^*) S_{\Theta}^{*n} g= (I-\Theta(z)\Theta(0)^{*})b_{n}$, whence 
	\begin{equation}
	\langle Af,g\rangle=\sum\limits_{n=0}^{\infty}[\langle S_{\Theta}^{n}B(I-\Theta(z)\Theta(0)^{*})a_{n},g\rangle+\langle f,S_{\Theta}^{n}B^{\prime}(I-\Theta(z)\Theta(0)^{*})b_{n}\rangle].
	\end{equation}
	
	Define then $\Phi, \Phi'\in H^2(\LL(E))$ by
	\[
	\Phi(z)x=B(I-\Theta(z)\Theta(0)^{*})x, \quad
	\Phi'(z)x=B'(I-\Theta(z)\Theta(0)^{*})x, \qquad x\in E.
	\] 
Then
	\[
	\langle Af,g\rangle=\sum\limits_{n=0}^{\infty}[\langle S_{\Theta}^{n}\Phi(z)a_{n},g\rangle+\langle f,S_{\Theta}^{n}\Phi^{\prime}(z)b_{n}\rangle].
	\]
	Using the formula $S_{\Theta}^{n}f=P_{\Theta} S^{n}f$ for $f\in \mathcal{K}_{\Theta}$,
	this becomes
	\begin{align*}
	\langle Af,g\rangle
	&=\sum\limits_{n=0}^{\infty}\langle P_{\Theta} S^{n} (\Phi(z)a_{n}),g\rangle+\sum\limits_{n=0}^{\infty}\langle f,P_{\Theta} S^{n}(\Phi^{\prime}(z)b_{n})\rangle\\
	&=\sum\limits_{n=0}^{\infty}\langle S^{n} (\Phi(z)a_{n}),P_{\Theta} g\rangle+\sum\limits_{n=0}^{\infty}\langle P_{\Theta} f,S^{n}(\Phi^{\prime}(z)b_{n})\rangle\\
	&=\sum\limits_{n=0}^{\infty}\langle  z^{n}\Phi(z)a_{n},g\rangle+\sum\limits_{n=0}^{\infty}\langle f,z^{n}\Phi^{\prime}(z)b_{n}\rangle\\
	&=\langle  \Phi(z)\sum\limits_{n=0}^{\infty}z^{n}a_{n},g\rangle+\langle f,\Phi^{\prime}(z)\sum\limits_{n=0}^{\infty}z^{n}b_{n}\rangle
	=\langle  \Phi f,g\rangle+\langle f,\Phi^{\prime}(z)g\rangle\\
	&=\langle  \Phi(z) f,g\rangle+\langle\Phi^{\prime *}(z) f,g\rangle
	=\langle  (\Phi(z)+\Phi^{\prime *}(z)) f,g\rangle=\langle A_{\Phi+\Phi^{\prime *}}f,g\rangle.
	\end{align*}
	Therefore $A=A_{\Phi+\Phi^{\prime *}}$, as claimed.
	
	\smallskip
Conversely, suppose that $A=A_{\Phi+\Phi^{\prime *}}$, with $\Phi, \Phi'\in H^2(\LL(E))$.For $f\in\KK_\Theta^\infty$ we have
	\begin{equation}\label{eq:A Phi+Phi'-...}
	(A_{\Phi+\Phi^{\prime *}}-S_{\Theta}A_{\Phi+\Phi^{\prime *}}S^{*}_{\Theta})f
	=(A_{\Phi}-S_{\Theta}A_{\Phi}S_{\Theta}^{*})f+(A_{\Phi^{\prime *}}-S_{\Theta}A_{\Phi^{\prime *}}S^{*}_{\Theta})f.
	\end{equation}

According to~\eqref{eq:semi-commutator for analytic}, the first term in~\eqref{eq:A Phi+Phi'-...} is
	\begin{equation*}
	(A_\Phi-S_\Theta A_\Phi S_\Theta^*)f= P_\Theta M_\Phi(I-SS^*) f.
	\end{equation*}

	Using the operator $\Omega$ defined by~\eqref{eq:omega}, we have
	\begin{equation}\label{eq:B Def}
	A_\Phi-S_\Theta A_\Phi S_\Theta^*= P_\Theta M_\Phi\Omega(I-S_\Theta S_\Theta^*)) P_\Theta=B(I-S_\Theta S_\Theta^*),
	\end{equation}
	where   $B=P_\Theta M_\Phi\Omega$.

	Similarly for $\Phi^{\prime}$, which is also analytic, 
	\begin{equation}\label{eq:B' Def}
		A_{\Phi^{\prime}}-S_{\Theta}A_{\Phi^{\prime}}S^{*}_{\Theta}=B'(I-S_\Theta S_\Theta^*), \qquad A_{\Phi^{\prime *}}-S_{\Theta}A_{\Phi^{\prime *}}S^{*}_{\Theta}=(I-S_\Theta S_\Theta^*) B'{}^*,
	\end{equation}
with $B'=P_{\Theta} M_{\Phi^{\prime}}\Omega$.	
	Using~\eqref{eq:B Def} and~\eqref{eq:B' Def} in~\eqref{eq:A Phi+Phi'-...}, we obtain
	\[
(	A_{\Phi+\Phi^{\prime *}}-S_{\Theta}A_{\Phi+\Phi^{\prime *}}S^{*}_{\Theta})f=B(I-S_\Theta S_\Theta^*)f+(I-S_\Theta S_\Theta^*) B^{\prime *}f,
	\]
which ends the proof of the theorem.
\end{proof}

\begin{rmk}\label{re:def B, B'}
The proof actually shows that for the MTTO $A_{\Phi+\Phi^{\prime *}}$ with $\Phi, \Phi'\in H^2(\LL(E))$  the operators $B,B'$ can be obtained as
\begin{equation}\label{eq: B and B'}
B=P_\Theta M_\Phi\Omega, \quad B'=P_{\Theta} M_{\Phi^{\prime}}\Omega.
\end{equation}

\end{rmk}

\begin{rmk}\label{re:alternate characterization}
	An application of the unitary operator $\tau$ defined by~\eqref{eq:tau} produces from Theorem 5.1 an alternate necessary and sufficient condition for the bounded operator $A$ to belong to $\mathcal{T}(\mathcal{K}_{\Theta}),$ namely  
	\begin{equation}\label{eq:S*AS}
	A-S_{\Theta}^{*}AS_{\Theta}=B(I-S_\Theta^* S_\Theta)+(I-S_\Theta^* S_\Theta) B^{\prime*}
	\end{equation}
for some operators $B, B^{\prime}$ from $\widetilde\DD$ to $\mathcal{K}_{\Theta}$.
	
	Indeed, one has to consider the operator $\tilde A =\tau A\tau^*$; then simple computations show that $\tilde{A}\in \MM\TT(\KK_{\tilde \Theta})$ if and only if $A\in\MM\TT(\KK_\Theta)$, and $\tilde{A}$ satisfies~\eqref{eq:quasicommutation}	on $\KK_{\tilde{\Theta}}$ if and only $A$ satisfies~\eqref{eq:S*AS} on $\KK_\Theta$.
\end{rmk}

As in the scalar case, one obtains from Theorem~\ref{th:rank two characterization of MTTO} a characterization of MTTOs by shift invariance.
For a bounded operator on $\mathcal{K}_{\Theta}$ we say that $A$ is \emph{shift invariant} if 
\begin{equation}\label{eq:shift invariance}
f, Sf\in \mathcal{K}_{\Theta}\text{ implies }  Q_{A}(f)=Q_{A}(Sf),
\end{equation}
where $Q_{A}$ is the associated quadratic form on $\mathcal{K}_{\Theta}$, defined by $Q_{A}(f)=\langle A f,f\rangle$ for $f\in \mathcal{K}_{\Theta}$.

Since $Sf\in K_\Theta$ if and only if $S_\Theta f=zf$, it follows from~\eqref{eq:action of S_Theta} that this is equivalent to $f\in\tilde\DD^\perp$, or to $(I-S_\Theta^* S_\Theta)f=0$. If $ f\in \KK^\infty_\Theta $, then $P_{\tilde\DD^\perp}f\in \KK^\infty_\Theta $. Therefore $\KK^\infty_\Theta \cap \tilde\DD^\perp$ is dense in $ \tilde\DD^\perp $, and the shift invariance condition~\eqref{eq:shift invariance} can be checked only for  $f\in \KK^\infty_\Theta \cap \tilde\DD^\perp$.

\begin{thm}\label{th:shift invariance}
	A bounded operator A on $\mathcal{K}_{\Theta}$ belongs to $\mathcal{T}(\mathcal{K}_{\Theta})$ if and only if $A$ is shift invariant.
\end{thm}

\begin{proof}
	Suppose that $A\in \mathcal{T}(\mathcal{K}_{\Theta}) $, so $A=A_{\Phi+\Phi^{\prime*}}$ for some $\Phi,\Phi^{\prime}\in H^{2}(\mathcal{L}(E)$. For $f\in \KK^\infty_\Theta \cap \tilde\DD^\perp$ we have
	\begin{equation*}
	\begin{split}
	Q_{A}(Sf)
&=\langle ASf,Sf\rangle=\langle A_{\Phi+\Phi^{\prime*}}Sf,Sf\rangle
=\langle A_{\Phi}Sf,Sf\rangle+\langle Sf,A_{\Phi^{\prime}}Sf\rangle\\
&=\langle M_{\Phi}Sf,P_\Theta Sf\rangle+\langle P_\Theta Sf,M_{\Phi^{\prime}}Sf\rangle.
	\end{split}
\end{equation*}
	From $f, Sf\in \mathcal{K}_{\Theta}$  it follows that $P_{\Theta} Sf=Sf=zf$. Therefore 
	\begin{align*}
	Q_{A}(Sf)
	&=\langle z\Phi f,zf\rangle+\langle zf,z\Phi^{\prime}f\rangle=\langle \Phi f,f\rangle+\langle f,\Phi^{\prime}f\rangle\\
	&=\langle A_{\Phi}f,f\rangle+\langle f,A_{\Phi^{\prime}}f\rangle=\langle A_{\Phi}f,f\rangle+\langle A^{*}_{\Phi^{\prime}}f,f\rangle\\
	&=\langle A_{\Phi}f,f\rangle+\langle A_{\Phi^{\prime *}}f,f\rangle=\langle A_{\Phi+\Phi^{\prime*}}f,f\rangle=Q_{A}(f).
	\end{align*}
	
	Conversely, suppose that the bounded operator $A$ on $\mathcal{K}_{\Theta}$ is shift invariant. We will prove that it satisfies  relation~\eqref{eq:S*AS}. Denote $\Delta=A-S_{\Theta}^{*}AS_{\Theta}$. If $f\in\tilde{\DD}^\perp$, then $S_\Theta f=S f$, and
		\begin{align*}
\langle \Delta f,f\rangle
	&=\langle Af,f\rangle-\langle S_{\Theta}^{*}AS_{\Theta}f, f\rangle =\langle Af,f\rangle-\langle AS_{\Theta}f,S_{\Theta} f\rangle\\
	&=\langle Af,f\rangle-\langle ASf,S f\rangle=Q_{A}(f)-Q_{A}(Sf)=0.
	\end{align*}
By the polarization identity we have $\langle \Delta f,g\rangle=0$ for $f,g\in \widetilde{\mathcal{D}}^{\perp}$. Thus the compression of $B$ to $\widetilde{\mathcal{D}}^{\perp}$ is the zero operator, or
\[
(I-P_{\widetilde{\DD}})\Delta (I-P_{\widetilde{\DD}})=0, \qquad
\Delta= (P_{\widetilde{\DD}}\Delta-\Delta)P_{\widetilde{\DD}} - P_{\widetilde{\DD}}\Delta.
\]

Using~\eqref{eq:tilde J}, we obtain
 \[
 \Delta= (P_{\widetilde{\DD}}\Delta-\Delta)\tilde{J}^* (I-S_{\Theta}^{*}S_{\Theta}) - (I-S_{\Theta}^{*}S_{\Theta})\tilde J\Delta
 \]
Therefore $A$ satisfies~\eqref{eq:S*AS}, so $A\in\MM\TT(\KK_\Theta)$.
\end{proof}

\begin{corollary}\label{co:weak closure}
	The space $\MTT(\KK_\Theta)$ is closed in the weak operator topology.
\end{corollary}

\begin{proof}
	Suppose the net $A_\alpha\in \MTT(\KK_\Theta)$ converges weakly to $A$. For $f\in\tilde{\DD}^\perp$ we have, by Theorem~\ref{th:shift invariance},
	\[
		\<A_\alpha Sf, Sf\>= \<A_\alpha f, f\>.
	\]
	Passing to the limit it follows that 
		\[
	\<A Sf, Sf\>= \<A f, f\>,
	\]
and the proof is finished by applying again Theorem~\ref{th:shift invariance}.
\end{proof}

\begin{rmk}\label{re:other operators}
	Theorem~\ref{th:shift invariance} allows us to obtain certain other classes of operators in $\MTT(\KK_\Theta)$. Suppose first that $X:\tilde{\DD}\to \DD$, and consider $\widehat{X}\in\LL(\KK_\Theta)$ defined by $ \widehat{X}f=X P_{\tilde{\DD}}f $. From~\eqref{eq:inclusions S(D), etc} it follows that $f\in\tilde{\DD}^\perp$ then $Sf\in\DD^\perp  $ and that  $ \widehat{X}^*f=X^* P_{\DD}f $. Therefore, if $ f\in\tilde{\DD}^\perp $, then
	\[
	Q_{\widehat{X}}(f)=\<Xf,f\>=0, 
	\quad Q_{\widehat{X}}(Sf)=\<\widehat{X}Sf,Sf\>
	=\<Sf,\widehat{X}^*Sf\>=0.
	\]
	Thus $\widehat{X}$ is shift invariant, hence in $\MTT(\KK_\Theta)$ by Theorem~\ref{th:shift invariance}.
The operators $\widehat{X}$ have finite rank; we will obtain a more general family of finite rank MTTOs in Section~\ref{se:finite rank}.

	Further on, the operator 
	\[
	 S_{\Theta, X}= S_\Theta P_{\tilde{\DD}^\perp}+\widehat{X}P_{\tilde{\DD}} =S_\Theta + (\widehat{X}-S_\Theta )P_{\tilde{\DD}}
	\]
	is also in $\MTT(\KK_\Theta)$. The operators $S_{\Theta, X}$ are called \emph{modified shifts}. For $X$ a contraction, they are precisely the perturbations of $S_\Theta$ considered in~\cite{Fu} (and, more generally, in~\cite{BL}). The case in which $X$ is unitary has been investigated at length in~\cite{Ma}; one obtains then vectorial analogues of the Clark unitary operators introduced in~\cite{Clark}.
\end{rmk}

\section{Condition for $A_{\Phi}=0$}

We start with the following statement, similar to the scalar case.

\begin{lem}\label{le:easy part of A_Phi=0}
	If $\Phi\in \Theta H^{2}(\mathcal{L}(E))+[\Theta H^{2}(\mathcal{L}(E))]^{*}$ then $A_{\Phi}=0$.
\end{lem}

\begin{proof} 
	Suppose $\Phi=\Theta\Phi_1+\Phi_2^*\Theta^*$, with $\Phi_1, \Phi_2\in \Theta H^{2}(\mathcal{L}(E))$, and $f\in \KK_\Theta^\infty$. Then 
	\[
	A_\Phi f= P_\Theta\Theta\Phi_1 f+P_\Theta\Phi_2^*\Theta^*f.
	\]
	Obviously $P_\Theta\Theta\Phi_1 f=0$. On the other side, if we take any $g\in \KK_\Theta^\infty$, then 
	\[
	\<P_\Theta\Phi_2^*\Theta^*f, g\>=
	\< \Phi_2^*\Theta^*f, g\>
	=\<f, \Theta\Phi_2 g\>=0.
	\]
	Therefore $A_\Phi=0$.
\end{proof}

We are interested to obtain the converse of this result. A first step is the next lemma.

\begin{lem}\label{le:A_Phi=0 for Phi analytic} Suppose  $A_{\Phi}=0$.
	\begin{itemize}
		\item[(i)]  If $\Phi\in H^{2}(\mathcal{L}(E))$, then $\Phi=\Theta \Phi_{1}$ for some $\Phi_{1}\in H^{2}(\mathcal{L}(E))$.
		
		\item[(ii)] If $\Phi^*\in H^{2}(\mathcal{L}(E))$, then $\Phi=(\Theta \Phi_{1})^*$ for some $\Phi_{1}\in H^{2}(\mathcal{L}(E))$.
	\end{itemize}
\end{lem}

\begin{proof} 
	Clearly it is enough to prove (i), since (ii) follows then by passing to the adjoint.

For any $x\in E$ the function 	$ (I-\Theta(z)\Theta(0)^{*})e_{j}\in\KK_\Theta$, and therefore
	\begin{equation*}
	0
	=A_{\Phi}(I-\Theta(z)\Theta(0)^{*})x\\
	=P_{\Theta} (\Phi(z)(I-\Theta(z)\Theta(0)^{*}))x.
	\end{equation*}
	The function $\Phi(z)(I-\Theta(z)\Theta(0)^{*})$ satisfies then the hypothesis of Lemma~\ref{le:from functions to matrices - Theta}. Therefore there exists $G\in H^2(\LL(E))$ such that
	\[
	\Phi(z)(I-\Theta(z)\Theta(0)^{*})=\Theta(z)G(z),
	\]
	or, noting that $(I-\Theta(z)\Theta(0)^{*})^{-1}\in H^\infty(\LL(E))$,
	\[
	\Phi(z)=\Theta(z)G(z)(I-\Theta(z)\Theta(0)^{*})^{-1}
	\]
	and $G(z)(I-\Theta(z)\Theta(0)^{*})^{-1}\in H^{2}(\mathcal{L}(E))$.
\end{proof}

The next result is the desired converse of Lemma~\ref{le:easy part of A_Phi=0}.

\begin{thm}\label{th:operator=0}
	If $A_\Psi=0$, then there exist $\Psi_1, \Psi_2\in \Theta H^{2}(\mathcal{L}(E))$ such that 
	$\Phi=\Theta\Phi_1+\Phi_2^*\Theta^*$.
\end{thm}

\begin{proof}
	Write $\Psi=\Phi+\Phi'{}^*$, with $\Phi, \Phi'\in H^2(\LL(E))$. Applying Theorem~\ref{th:rank two characterization of MTTO}, it follows (see also Remark~\ref{re:def B, B'}) that if $\Omega$ is defined by~\eqref{eq:omega} and $B, B'$ are given by~\eqref{eq: B and B'}, then
	\begin{equation}\label{eq:A_Phi=0 implies a relation with B, B'}
	B(I-S_\Theta S_\Theta^*))+(I-S_\Theta S_\Theta^*))B'=0.
	\end{equation}
	Therefore the range of $B(I-S_\Theta S_\Theta^*))$ is also contained in $\DD$, so 
	\begin{equation}\label{eq:op=0-1}
		B(I-S_\Theta S_\Theta^*))=P_\DD M_\Phi \Omega P_\DD =P_\Theta M_\Phi \Omega P_\DD=P_\DD M_\Phi .
	\end{equation}

	Consider the map $\chi:\LL(E)\to \LL(\DD)$ defined by
	\[
	\chi(T)= P_\DD M_T \Omega P_\DD,
	\]
	$M_T$ being  multiplication by the constant operator $T$. We claim that $\chi$ is one-to-one. Indeed, suppose $\chi(T)=0$. Since $\Omega$ is invertible from $\DD$ to $E$, it follows that $P_\DD M_T=0$, and so
	\[
	Tx\perp (I-\Theta(z)\Theta(0)^*)y
	\]
	for all $x,y\in E$. In particular, 
	\[
	Tx\perp (I-\Theta(z)\Theta(0)^*)Tx=P_\Theta Tx.
	\]
	Therefore $Tx\in\Theta H^2(E)$. From Lemma~\ref{le:from functions to matrices - Theta} it follows then that
	\[
	T=\Theta G
	\]
	for some $G\in H^\infty(\LL(E))$. If $T$ is not identically 0, this contradicts Lemma~\ref{le:inner functions}.
	
	Being a one-to-one map between spaces of the same dimension $d^2$, $\chi$ is also onto. Therefore there exists a constant matrix $\Phi_0$ such that
	\[
	P_\DD M_{\Phi_0} \Omega P_\DD= 	P_\DD M_{\Phi} \Omega P_\DD
	\]
	and thus
	\begin{equation}\label{eq:A=0 1}
		P_\DD M_{\Phi-\Phi_0} \Omega P_\DD=0.
	\end{equation}

Recall now from~\eqref{eq:op=0-1} that $P_\DD M_\Phi\Omega P_\DD= P_\Theta M_\Phi\Omega P_\DD$. On the other hand, the values of $M_{\Phi_0} \Omega P_\DD$ are constant functions, and so  
\[
P_\DD M_{\Phi_0} \Omega P_\DD=
P_\Theta M_{\Phi_0} \Omega P_\DD.
\]
It follows then from~\eqref{eq:A=0 1} that
\[
P_\Theta M_{\Phi-\Phi_0} \Omega P_\DD=0,
\]
or
\[
P_\Theta M_{\Phi-\Phi_0} x=0
\]
for any $x\in E$. By Lemma~\ref{le:from functions to matrices - Theta} there exists $\Phi_1\in H^2(\LL(E))$ such that $\Phi-\Phi_0=\Theta\Phi_1$.

In particular,  $A_{\Phi-\Phi_0}=0$, and therefore, since $A_\Psi=A_{\Phi+\Phi'{}^*}=0$, we also have $A_{(\Phi'+\Phi_0^*)^*}=0$. Applying Lemma~\ref{le:A_Phi=0 for Phi analytic} (ii), it follows that there exists $\Phi_2\in H^2(\LL(E))$ such that $(\Phi'+\Phi_0^*)^*=\Phi_2^*\Theta^*$. Then
\[
\Phi=(\Phi-\Phi_0)+(\Phi'+\Phi_0^*)^*=\Theta\Phi_1+\Phi_2^*\Theta^*,
\]
which finishes the proof of the theorem.
\end{proof}

As a corollary, we show that every MTTO has a symbol in a certain class. Denote by $\MM_\Theta$ the orthogonal complement of $ \Theta H^2(\LL(E)) $ in $ H^2(\LL(E)) $ endowed with the Hilbert--Schmidt norm. It is easy to see that in a given basis the matrices of functions in $\MM_\Theta$ are characterized by the fact that the columns are functions in~$ \KK_\Theta $.

\begin{corollary}\label{co:standard symbols}
	For any $A\in\MTT(\KK_\Theta) $ there exist $ \Psi_1, \Psi_2\in \MM_\Theta $ such that $ A=A_{\Psi_1+\Psi_2^*} $. If $ \Psi_1', \Psi_2'\in \MM_\Theta $ also satisfy $ A=A_{\Psi'_1+\Psi'_2{}^*} $, then $\Psi_1'=\Psi_1+ k^\Theta_0 X, \Psi_2'=\Psi_2-(k^\Theta_0 X)^* $, with $ X\in\LL(E) $.
\end{corollary}

\begin{proof}
	Since $L^2(\LL(E))= \MM_\Theta + \Theta H^2(\LL(E)) +\MM_\Theta^* + (\Theta H^2(\LL(E)))^*  $, the first assertion follows by decomposing $\Phi$ accordingly and using Theorem~\ref{th:operator=0}.  
	
	For the second part of the corollary, since 
	\[
	A_{(\Psi_1-\Psi'_1)+(\Psi_2^*-\Psi'_2{}^*)}= 0,
	\]
	and $\Psi_1-\Psi'_1, \Psi_2-\Psi'_2\in\MM_\Theta$,
	it is enough to show that
	\begin{equation}\label{eq:K, K*, Theta, Theta*}
	(\MM_\Theta+\MM_\Theta^*)\cap (\Theta H^2(\LL(E))+(\Theta H^2(\LL(E)))^*)=\{k^\Theta_0 X-(k^\Theta_0 X)^*: X\in\LL(E)\}.
	\end{equation}
	
	Suppose then that the functions $F, G\in\MM_\Theta$, $F_1, G_1\in H^2(\LL(E))$ satisfy
	\[
	F+G^*=\Theta F_1+(\Theta G_1)^*.
	\]
	But we have
	\[
	P_{\MM_\Theta}(F+G^*)= F+ P_{\MM_\Theta}(G^*)=
	F+ P_{\MM_\Theta}(G(0)^*)
	=F+(I-\Theta \Theta(0)^*)G(0)^*,
	\]
	and
	\[
	P_{\MM_\Theta}(\Theta F_1+(\Theta G_1)^*)
	=P_{\MM_\Theta}((\Theta G_1)^*)=
	(I-\Theta \Theta(0)^*)(\Theta(0) G_1(0))^*.
	\]
	Comparing the last  equations, we obtain that
$
	F=(I-\Theta \Theta(0)^*)X
	$
	for some $X\in\LL(E)$. Similarly, $ G=(I-\Theta \Theta(0)^*)Y $ for some $Y\in\LL(E)$. Now, since $ F+G^*\in \Theta H^2+(\Theta H^2)^* $, it follows that the constant matrix $ X+Y^* $ is in $ \Theta H^2+(\Theta H^2)^* $. 
	
	Suppose $ X+Y^* \not=0$. Then there are nonzero functions $F_2, G_2\in \Theta H^2$ such that $ X+Y^*= \Theta F_2+ (\Theta G_2)^* $, or
	\[
	(\Theta G_2)^*=X+Y^*-\Theta F_2.
	\]
	Since the left hand side is coanalytic and the right hand side is analytic, both must be constant (and nonzero). Thus $ \Theta G_2 $ is constant, which contradicts Lemma~\ref{le:inner functions}. Therefore $X=-Y^*$, whence~\eqref{eq:K, K*, Theta, Theta*} is satisfied.
\end{proof}

It is well known (see, for instance,~\cite[Chapter 2]{Pe}), that the space $\KK_\Theta$ is finite dimensional if and only if $\Theta$ is a finite Blaschke--Potapov product. In this case we may use the previous corollary to obtain the dimension of the space $\MTT(\KK_\Theta)$.

\begin{corollary}\label{co:dim mtto}
	If $\dim\KK_\Theta=n$, then $\dim\MTT(\KK_\Theta)=2n^d-d^2$.
\end{corollary}

\begin{proof}
	First, it is immediate that $\dim\MM_\Theta=(\dim \KK_\Theta)^d=n^d$.
Consider then the linear map $L:\MM_\Theta\times\MM_\Theta\to \MTT(\KK_\Theta)$ defined by
\[
L(\Phi_1, \Phi_2)=A_{\Phi_1+\Phi_2^*}.
\]
According to Corollary~\ref{co:standard symbols}, $L$ is onto, while 
\[
\ker L=\{k^\Theta_0 X -(k^\Theta_0 X)^* :X\in\LL(E)  \}.
\]
The proof is finished by noting that  $\dim(\MM_\Theta\times\MM_\Theta)=2n^d$ and $\dim\ker L=\dim\LL(E)=d^2$.	
	\end{proof}

\section{A class of finite rank operators}\label{se:finite rank}

As noted in Section~\ref{se:model space}, for any $ x\in E $
we have $ k^\Theta_\lambda x\in \KK_\Theta$ and $ \widetilde{k^\Theta_\lambda}\in \KK_\Theta $. Therefore
the matrix valued analytic functions $ k^\Theta_\lambda $ and $ \widetilde{k^\Theta_\lambda} $ may be considered as bounded operators from $ E $ to $ \KK_\Theta $. To avoid any confusion, we will denote these operators by $ K_\lambda, \widetilde{K}_\lambda: E\to \KK_\Theta $; therefore $ K_\lambda^*, \widetilde{K}_\lambda^*:  \KK_\Theta\to E $. With these notations,
the relations in Lemma~\ref{le:formulas for S_Theta on k lambda} become equalities between operators from $E$ to $ \KK_\Theta $:
\begin{equation}\label{eq:new K_lambda}
S_\Theta K_\lambda=\frac{1}{\bar{\lambda}}K_\lambda- \frac{1}{\bar{\lambda}}K_0,\quad
S_\Theta \widetilde{K}_\lambda= \lambda \widetilde{K}_\lambda - K_0\Theta(\lambda).
\end{equation}

 We  obtain then a class of finite rank operators in 
 $\MTT(\KK_\Theta)$.

\begin{thm}\label{th:rank one1} For any $Y\in\LL(E)$ and $\lambda\in \mathbb{D}$ the operators 
	$ K_\lambda Y \widetilde{K}_\lambda^* $ and $ \widetilde{K}_\lambda Y  K_\lambda^* $ have rank equal to the rank of $ Y$ and
	belong to $\MTT(\KK_\Theta)$.
\end{thm}

\begin{proof} Obviously
	it suffices to consider the first operator. We  apply Theorem~\ref{th:rank two characterization of MTTO}. Assuming $\lambda\not=0$ and using~\eqref{eq:new K_lambda}, we obtain
		\begin{align*}
	K_\lambda Y \widetilde{K}_\lambda^*- S_\Theta K_\lambda Y \widetilde{K}_\lambda^* S_\Theta^*
	&=K_\lambda Y \widetilde{K}_\lambda^*
	-\Big(\frac{1}{\overline{\lambda}}K_\lambda -\frac{1}{\overline{\lambda}}
	K_0\Big)Y
	\big(\bar{\lambda}\widetilde{K}_\lambda^*- \Theta(\lambda)^* K_0^*
	\big)\\
	&= K_0Y \widetilde{K}_\lambda^* + \frac{1}{\overline{\lambda}}
(K_\lambda-K_0)Y  \Theta(\lambda)^*K_0^*
	\end{align*}

	But the range of $K_0$ is contained in $\DD=\image(I-S_{\Theta}S_{\Theta}^{*})$, and $ I-S_{\Theta}S_{\Theta}^{*} $ is invertible on $ \DD $; if we denote that inverse by $Z$, we have  $K_0=(I-S_\Theta S_\Theta^*))ZK_0$, $K_0^*=K_0^*Z^*(I-S_\Theta S_\Theta^*))$, and
	\[
	\begin{split}
	K_\lambda Y \widetilde{K}_\lambda^*- S_\Theta K_\lambda Y \widetilde{K}_\lambda^* S_\Theta^*&=
	(I-S_\Theta S_\Theta^*))ZK_0Y \widetilde{K}_\lambda^*\\&\qquad +
	\frac{1}{\overline{\lambda}}
	(K_\lambda-K_0)Y  \Theta(\lambda)^*K_0^*Z^*(I-S_\Theta S_\Theta^*)).
	\end{split}
	\]
	Equation~\eqref{eq:quasicommutation} is therefore satisfied by taking $B=	\frac{1}{\overline{\lambda}}
	(K_\lambda-K_0)Y  \Theta(\lambda)^*K_0^*Z^*$, $ B'=\widetilde{K}_\lambda Y^*K_0^*Z^* $. The proof in the case $\lambda\not=0$ is concluded by invoking Theorem~\ref{th:rank two characterization of MTTO}. For $\lambda=0$, we may note that $K_\lambda\to K_0$ and $ \widetilde{K}_\lambda\to \widetilde{K}_0 $ weakly when $\lambda\to 0$, and use Corollary~\ref{co:weak closure}.
	
	The assertion concerning the rank is left to the reader.
\end{proof}

We have thus obtained a class of finite rank MTTOs. For $\lambda=0$, they are precisely the operators $ \widehat{X} $ defined in Remark~\ref{re:other operators}.

\begin{rmk}
In case $Y$ has rank 1, say $Y=x\otimes y$, we have 
\[
K_\lambda Y \widetilde{K}_\lambda^*= k^\Theta_\lambda x\otimes\widetilde{k_{\lambda}^{\Theta}}y
\]
and we obtain a family of rank one operators similar to the scalar case. 
As in the scalar case, we may obtain supplementary operators of rank one in case $ k^\Theta_\lambda x $ and $ \widetilde{k^\Theta_\lambda y} $ have limits in $\KK_\Theta$ when $ \lambda $ tends nontangentially to a point $ \mu $ on the unit circle. 
One can show that this is equivalent to the conditions
\[
\begin{split}
&\liminf_{\lambda\to\mu}\frac{1}{1-|\lambda|^2} (\|x\|^2-\|\Theta(\lambda)^*x\|^2) <\infty,\\
&\liminf_{\lambda\to\mu}\frac{1}{1-|\lambda|^2} (\|y\|^2-\|\Theta(\lambda)y\|^2) <\infty.
\end{split}
\]
Moreover,  this procedure allows one to obtain all rank one operators in $\MTT(\KK_\Theta)$.
The proof is rather tedious and will be presented elsewhere.
\end{rmk}

\end{document}